\documentclass[reqno, english]{amsart}
\usepackage{etex}
\usepackage{amsmath,amssymb,amsthm,bbm,mathtools,comment}
\usepackage[shortlabels]{enumitem}
\usepackage[pdftex,colorlinks,backref=page,citecolor=blue]{hyperref}
\hypersetup{pdfpagemode=UseNone,pdfstartview={XYZ null null 1.00}}
\usepackage[mathscr]{euscript}
\usepackage[usenames,dvipsnames]{color}
\usepackage{adjustbox,tikz,calc,graphics,babel,standalone}
\usetikzlibrary{shapes.misc,calc,intersections,patterns,decorations.pathreplacing}
\usetikzlibrary{arrows,shapes,positioning}
\usetikzlibrary{decorations.markings}
\usepackage[final]{microtype}
\usepackage[numbers]{natbib}
\usepackage{cmtiup}
\usepackage{amsfonts}
\usepackage{graphicx}
\usepackage{caption}
\usepackage{subcaption}
\usepackage{verbatim}
\usepackage{array}
\usepackage[frame,cmtip,arrow,matrix,line,graph,curve]{xy}
\usepackage{graphpap, color, pstricks}
\usepackage{pifont}
\usepackage[final]{microtype}
\usepackage{cmtiup}

\allowdisplaybreaks

\theoremstyle{plain}
\newtheorem{theorem}{Theorem}[section]		
\newtheorem{lemma}[theorem]{Lemma}

\newtheorem{proposition}[theorem]{Proposition}

\newtheorem{problem}[theorem]{Problem}
\theoremstyle{remark}

\def\CC{\mathcal{C}}

\newcommand{\eps}{\ensuremath{\varepsilon}}

\let\originalleft\left
\let\originalright\right
\renewcommand{\left}{\mathopen{}\mathclose\bgroup\originalleft}
\renewcommand{\right}{\aftergroup\egroup\originalright}

\makeatletter
\def\imod#1{\allowbreak\mkern10mu({\operator@font mod}\,\,#1)}
\makeatother

\newcommand{\hide}[1]{}

\author{Ant\'onio Gir\~ao}
\address[Gir\~ao]{Mathematical Institute, University of Oxford, Andrew Wiles Building, Radcliffe Observatory Quarter, Woodstock Road, Oxford, UK.}
\email{girao@maths.ox.ac.uk}

\author{Zach Hunter}
\address[Hunter]{Department of Mathematics, ETH Z\"urich, Switzerland.}
\email{zach.hunter@ifor.math.ethz.ch}

\begin{document}

\title{Monochromatic odd cycles in edge-coloured complete graphs}
\begin{abstract}
    It is easy to see that every $q$-edge-colouring of the complete graph on $2^q+1$ vertices must contain a monochromatic odd cycle. A natural question raised by Erd\H{o}s and Graham in $1973$ asks for the smallest $L(q)$ such that every $q$-edge-colouring of $K_{2^q+1}$ must contain a monochromatic odd cycle of length at most $L(q)$. In here, we show that $L(q)=O\left(\frac{2^q}{q^{1-o(1)}}\right)$ giving the first non-trivial upper bound on $L(q)$.   
\end{abstract}
 
 \maketitle
 
\section{Introduction}
Ramsey theory is a branch of combinatorics with a rich history which has seen many developments in the last couple of years (see e.g~\cite{Ramsey1, Ramsey2}, or the recent notes \cite{yuvalnotes}).
Typically, it is concerned with finding monochromatic structures in every finite colouring of a rich enough space. In almost all instances, our quantitative understanding of these questions is very poor as the number of colours grows. Perhaps the most notorious question here is the \textit{Schur-Erd\H{o}s problem}, which asks if $R_q(K_3)$ (the smallest $N$ so that an $q$-edge-colouring $K_N$ contains a monochromatic triangle) grows exponentially in $q$. Furthermore, for every $\ell\geq 1$, we do not know the growth order of $R_q(C_{2\ell+1})$ as $q$ grows for fixed $\ell$ (the best upper bound being of the form $O(q^{q/2})$ for $\ell\ge 2$, cf. \cite{lin-chen}). 

The corresponding case for even cycles is much easier, well-known density results give that $R_q(C_{2l}) \le O_l(q^{2+2/l})$. The other instance where we now have good understanding is when one fixes $q$ and takes $l$ sufficiently large -- here a wonderful result of Jenssen and Skokan shows that $R_q(C_{2l+1}) = (l-1)2^{q-1}+1$ \cite{JS}. The problem we study in this note is perhaps closer to the first two multicolour ramsey questions about cycles, as the parameter that will be growing is the number of colours, and it is not at all clear what the actual answer should conjecturally be.

It is a simple exercise to see that every complete graph on $n$ vertices has a $q$-edge-colouring such that each colour consists of a bipartite graph, if and only if $n\leq 2^{q}$. Based on this observation, Erd\H{o}s and Graham~\cite[Question~(iii) in Section~6]{EG} asked the following natural question (cf. \url{https://www.erdosproblems.com/609}). 
\begin{problem}
 Given $q\ge 1$, let $L(q)$ be the smallest integer $\ell$ so that for $q$-edge-colouring of $K_{2^q+1}$, there exists a monochromatic cycle of odd length at most $\ell$. How does $L(q)$ grow?
\end{problem}
\noindent Day and Johnson \cite[Theorem~2]{DJ}, answering a question of Chung~\cite[Problem~75]{Chung} showed that $L(q)$ is unbounded and more precisely they proved $L(q)\geq 2^{\Omega(\sqrt{\log q})}$ (see \cite[Corollary~6]{DJ} for the quantitative bound). 

The main result of this short note is the following upper bound on $L(q)$.
\begin{theorem}\label{thm: main}
    For every $\varepsilon>0$, there is $q_0>0$ such that in every $q$-edge-colouring of $K_{2^q+1}$ there is a monochromatic odd cycle of length at most $\frac{2^{q}+1}{q^{1-\varepsilon}}$, provided $q>q_0$. 
\end{theorem} \noindent To the best of the authors' knowledge, no bound of the form $L(q) = o(2^q)$ was previously in the literature.

\section{Preliminary lemmas}

We first require two simple lemmas. 
\begin{lemma}\label{Lem:1}
    Let $G$ be a graph on $n$ vertices which does not contain an odd cycle of length at most $2k+1$, for some $k\ge \log_2 n$. Then, there is a set $S\subset V(G)$ of size at most $\frac{\log_2 n}{k}\cdot n$ that $G- S$ is bipartite and each connected component of $G-S$ has radius at most $k$. 
\end{lemma}
\begin{proof}
    Write $\eps := \frac{\log_2 n}{k}$. We quickly note that $(1+\eps)^k \ge n$. Recall Bernoulli's inequality, which says $(1+x)^r\ge 1+xr$ for all real $x>-1,r\ge 1$. We apply said inequality with $x:= \eps,r:= 1/\eps$ (since we assumed $k$ is large, we get $1/\eps\ge 1$), giving $(1+\eps)^k = ((1+\eps)^{1/\eps})^{\log_2 n} \ge 2^{\log_2 n} =n$ as desired. 

    Now pick a vertex $x$, and let $N^{(i)}(x)$ be the set of vertices at distance $i$ from $x$, for some $i\in \mathbb{N}$ and $N^{\leq i}(x)$ the set of vertices at distance at most $i$ from $x$. If $|N^{\leq i}(x)|> (1+\eps)|N^{\leq i-1}(x)| $, for all $i\in [k]$, we would get a contradiction, since $|N^{\leq k}(x)|> (1+\eps)^k\ge n$. So there is some $j\leq k $ with $|N^{(j)}(x)|\leq n^{1/k} |N^{\leq j-1}(x)|$. Observe that by assumption $N^{\leq j-1}(x)$ is bipartite since otherwise it would contain a cycle of length at most $2j-1$. We add $N^{(j)}(x)$ to $S$ and delete the vertices in $N^{\le j}(x)$ from $G$.
    
    We then iterate this procedure. Consider some $x$ in the remaining graph. Again, we can find $j\le k$ so that $|N^{\leq j}(x)|\le (1+\eps)|N^{\leq j-1}(x)|$; and as before we add the $N^{(j)}(x)$ to $S$ and delete $N^{\leq j}(x)$ from $G$. We continue until no vertices remain. Throughout the entire process, we have that $|S|$ is at most an $\eps$-fraction of the number of vertices we have deleted thus far. So when we stop $|S|\le \eps n$, and by construction $G-S$ has the desired properties.  
\end{proof}

\begin{lemma}\label{Lem:2}
    Let $F$ be a non-bipartite graph. Suppose $H\subset F$ is some subgraph where every component has radius at most $r$. Finally, suppose $H'$ is some subgraph of $H$ with at most $m$ connected components. Then $F$ contains an odd cycle of length at most $|V(F)\setminus V(H')|+(4r+1)m$.
 \end{lemma}
 \begin{proof}
     Let $\mathcal{C}$ be the shortest odd cycle in $F$. If $\mathcal{C}$ contains at most $4r+1$ vertices from each connected component in $H'$ we are done. Suppose not, then there must be a connected component $X$ in $H$ such that $|V(\mathcal{C})\cap V(X)|\geq 4r+2$. 
     
     This implies there are two vertices $x,y\in \mathcal{C}\cap V(X)$ whose distance in $\mathcal{C}$ is at least $2r+1$. Next let $P$ be a path in $H[X]$ between $x,y$ of length at most $2r$ (which exists because the component has radius at most $r$). We claim $\mathcal{C}\cup P$ must contain an odd cycle of length strictly smaller than $|\mathcal{C}|$, giving a contradiction.

     One finds the shorter odd cycle as follows. Since $\CC$ is an odd cycle, there is a path $P'\subset \CC$ from $x$ to $y$, so that $|P'|$ and $|P|$ have the same parity. We delete the edges from $P'$ and add the edges from $P$. This gives us an odd circuit $\mathcal{C}'$ with length strictly smaller than that of $\mathcal{C}$ (since $|P'|\ge 2r+1>|P|$). Any odd circuit of length $\ell'$ certainly contains an odd cycle of length $\ell''\le \ell'$, establishing the claim.
 \end{proof}
We turn now to the main ingredient of the proof. 
 \begin{lemma}\label{Lem: 3}
    Let $n\geq 2^q/2$ be an integer and $\varepsilon>1/q$. Let $(A_1,B_1),\ldots (A_q,B_q)$ be pairs of disjoint subsets of $[n]$ and suppose that for every $i$, $|A_i|+|B_i|\leq (1-\varepsilon)n$. 
    Then, there is a set $L\subset [n]$ of size at least $2^{\varepsilon q}/2$ such that no edge $e\in L^{(2)}$ is of the form $e=ab$, for some choice of $i\in [q]$ and $a\in A_i,b\in B_i$. 
\end{lemma}
\begin{proof}
 For a vertex $x\in [n]$, denote $d(x)\coloneqq \{i\in [q]: i\in A_i\cup B_i\}$ and let $d\coloneqq \sum_{x\in [n]} d(x)/n$.  
By assumption on the sizes we have $d\leq (1-\varepsilon)q$. Now, for every $i$ take uniformly at random either the set $A_i$ or $B_i$ independently. Let $U$ be the union of these sets. 
The expected size of $[n]\setminus U$ is at least  $$\sum_{x\in [n]} 2^{-d(x)}\geq n2^{-d}\geq n\cdot 2^{-(1-\varepsilon) q}\geq 2^{\varepsilon q}/2.$$ Hence there is a choice of $U$ such that $L\coloneqq [n]\setminus U$ has size at least $2^{\varepsilon q}/2$. Note for each $i\in [q]$, we have that either $L\cap A_i$ or $L\cap B_i$ is empty, meaning each edge $e\in L^{(2)}$ behaves as desired.
\end{proof}

\section{Proof of main theorem}

In what follows we will consider some edge-colouring of $K_n$ using colours in $[q]$. We will use $G^{(i)}$ to denote the subgraph with all edges of colour $i$. Often, given a graph $G$ and a set of vertices $S$, we write $G-S$ to denote the graph induced by deleting the vertices belonging to $S$.
 
Let $\varepsilon>0$ be fixed. Let $C(\varepsilon)$ be a large enough positive constant. We shall prove that $L(q)\le C(\varepsilon) \frac{2^q}{q^{1-\varepsilon}}$ by induction on $q$ (which clearly implies Theorem~\ref{thm: main}). For small $q$, there is nothing to show as $C(\varepsilon)\frac{2^{q}}{q^{1-\varepsilon}}\geq 2^{q}+1$.
Let $n\coloneqq 2^{q}+1$ and suppose $K_n$ is $q$-edge-coloured. By induction, we may assume every colour class contains an odd cycle otherwise we can pass to a subgraph of size $\geq 2^{q-1}+1$ which is $(q-1)$-edge-coloured. By induction, we get a monochromatic odd cycle of size at most $C(\varepsilon)\frac{2^{q-1}}{(q-1)^{1-\varepsilon}}\leq C(\varepsilon)\frac{2^{q}}{q^{1-\varepsilon}}$, as we wanted to show.

Firstly, we apply Lemma~\ref{Lem:1} to each colour class with $k\coloneqq 8q^3$ (by assumption, for $q$ large, $G^{(i)}$ should not contain an odd cycle of length at most $2k+1$). For each colour $i$, let $S_i$ be the set output by the lemma and take $S\coloneqq \cup_{i=1}^{q} S_i$. By assumption $|S|\leq q\frac{2(2^{q}+1)\log(2^{q}+1)}{8q^3}\leq \frac{n}{2q}$. We shall imagine deleting the vertices in $S$, and find a short odd cycle by consider the bipartite edge-coloring on the remaining vertices.  

For each colour $i\in [q]$, we do the following. Let $X_{1},\ldots X_m$ be the components of the bipartite graph $G^{(i)}-S$.
Then, define $\mathcal{B}_i=\{X_1,\ldots , X_t\}$ to be the set of these components of size at most $4q^{10}$. 

Suppose $|V(\mathcal{B}_i)|\leq n/q^{1-\varepsilon}$. We now apply Lemma~\ref{Lem:2} to $G^{(i)}$. Indeed, $H' \coloneqq G^{(i)}- (S\cup V(\mathcal{B}_i))$ consists of at most $m\leq n/4q^{10}$ components, and $H\coloneqq G^{(i)}-S_i\supset H'$ only has components with radius at most $8q^3$. So applying Lemma~\ref{Lem:2} and recalling $G^{(i)}$ was assumed to have an odd cycle, it must have one of size at most $|S|+|\mathcal{B}|+(32q^3+2) \frac{n}{q^{10}}\leq \frac{2n}{q^{1-\varepsilon}}$.

We may therefore assume that for all $i\in [q]$, $|V(\mathcal{B}_i)|> n/q^{1-\varepsilon}$. Here, we shall now apply Lemma~\ref{Lem: 3}. First, we delete all vertices in $S$, leaving us with a vertex set $V'$ of $n'\ge 2^q/2$ vertices left, and write $H^{(i)}\coloneqq G^{(i)}-S$ (similar to before). For each $i\in [q]$, the graph $ H^{(i)}-V(\mathcal{B}_i)$ is bipartite, so we can fix some partition of these vertices into independent sets $A_i,B_i$. 
We have from assumptions that $|A_i|+|B_i|\leq (1-\delta)n'$, where $\delta> \frac{1}{q^{1-\varepsilon}}$. Therefore (by Lemma~\ref{Lem: 3}), there is a set $L\subset V'$ of size at least $2^{\delta q}/2$, where no edge belongs to any of the graphs $H^{(i)}-V(\mathcal{B}_i)$ (for $i\in [q]$). Hence, all the edges in $L$ must be covered using the connected components from $\mathcal{B}_1,\dots,\mathcal{B}_q$. But this is not possible since the maximum degree of the union of these small components is at most $q\cdot 4q^{10}$ which is smaller than $ 2^{\delta q}/2 $ for large $q$, giving a contradiction.

\section{Concluding remarks}
The current bounds on $L(q)$ are very far apart. As mentioned in the introduction, Day and Johnson \cite{DJ} showed that $L(q)\ge \exp(\Omega(\sqrt{\log q}))$. They were the first to show that $L(q)$ is super-constant, and this remains the state of the art. While we have focused on $K_{2^q+1}$, one can obviously generalize to consider when there are more vertices.

Let $L(q,N)$ be such that any $q$-edge-coloring of $K_N$ still has a monochromatic odd cycle of length $\ell\le L(q,N)$. We note that a simple product colouring argument \cite[Lemma~7.1]{DJ} combined with with the original argument of Day and Johnson gives a lower bound of $L(q,2^{q+1}) \ge \Omega(\sqrt{\log q})$. Meanwhile, in this setting one can get much better upper bounds.

Indeed, we get the following. 
\begin{proposition}
    For $q\ge 1,\delta \in (0,1)$, we have that $L(q,(1+\delta)2^q)\le O(q^2 \delta^{-1})$. 
\end{proposition}

\begin{proof}
    Write $N:= (1+\delta)2^q< 2^{q+1}$. Fix some $q$-edge-colouring of $K_N$, and for $i\in [q]$ let $G^{(i)}$ denote the graph in colour $i$.
    
    Simply take $k:= \lceil 2q(q+1)\delta^{-1}\rceil$, and suppose for sake of contradiction that each $G^{(i)}$ has no odd cycles with length at most $2k+1$. As $\log_2 N< q+1\le k$, we get that we can apply Lemma~\ref{Lem:1}, deleting a set $S$ of at most $q\cdot \frac{\log_2N}{ k}<\frac{\delta}{2}N$ vertices making $G^{(i)}-S$ bipartite for every $i\in [q]$. But $K_N-S$ has more than $2^q$ vertices, so by pigeonhole and bipartiteness there must be distinct $x,y\not\in S$ that belong to the same independent set in each $G^{(i)}-S$, meaning the edge $xy$ was not coloured (contradicting that $E(G^{(1)})\cup \dots \cup E(G^{(q)})=E(K_N)$).
\end{proof}

Another result of Day and Johnson shows that for each $k$, there exists $\eps_k>0$ so that $L(q,(2+\eps_k)^q)>2k+1$ \cite[Theorem~4]{DJ}. If one inspects the proof of Lemma~\ref{Lem:1}, then when $k<\log_2 n$ one still gets that $G-S$ is bipartite for some set $S$ of size at most $(1-n^{-1/k})n$. A simple corollary then is that any $n$-vertex graph $G$ with no odd cycle of length at most $2k+1$ has an independent set of size at least $(1/2)(n-n^{-1/k})$, giving the bound $L(q,(2+\eps)^q)\le C_{\eps}q +O_\eps(1)$ by a simple induction (for the constant $C_{\eps}$ satisfying $(2+\eps)^{-2/C_{\eps}} =1-\frac{2}{2+\eps}$). 

\hide{Finally, it would be very interesting to better understand what is happening for $L(q,C^q)$ for large values of $C$. As mentioned before, it is a famous open question whether the multicolour ramsey number of the triangle is exponential, which is equivalent to asking if $L(q,C^q)=3$ when $C$ is some large absolute constant. The astute reader may observe that the upper bound mentioned in the last paragraph does not yield improved bounds (and starts giving worse ones) when one tries to bound $L(q,100^q)$ and $L(q,1000^q)$. One can get around this by taking small $\delta,\eta>0$, considering huge $C$, and taking $k=\eta q$ roughly $\delta q$ times, to obtain some recurrence of the form $L(q,C^q)\le \max\{2\eta q+1, L((1-\eta)q, C^q (2C^{1/\eta})^{-\delta q}\} $. So taking $\delta= \eta/2$ (say), if we define $f(C) := \limsup_{q\to\infty} \frac{L(q,C^q)}{q}$, one gets the recurrence $f(C)\le \max\{2\eta,f(\sqrt{C}/2)\}$. Surely the correct way to carry out these bounds is via some differential equation methods, but we will not do this here. While the focus of our work was on $L(q,2^q+1)$, we hope to see further progress in these other regimes. This includes the lower bounds -- Day and Johnson speculate (see the bottom of Section~2 in \cite{DJ}) that their bounds should not be sharp, but several years later they still remain the state of the art.}

\bibliographystyle{alpha}
\bibliography{bibliography}

\newcommand{\etalchar}[1]{$^{#1}$}
\begin{thebibliography}{CGMS}

\bibitem[BBC{\etalchar{+}}]{Ramsey2}
P.~Balister, B.~Bollob\'as, M.~Campos, S.~Griffiths, E.~Hurley, R.~Morris,
  J.~Sahasrabudhe, and M.~Tiba.
\newblock Upper bounds for multicolour {R}amsey numbers.
\newblock {\em ar{X}iv:2410.17197}.

\bibitem[CGMS]{Ramsey1}
M.~Campos, S.~Griffiths, R.~Morris, and J.~Sahasrabudhe.
\newblock An exponential improvement for diagonal {R}amsey.
\newblock {\em ar{X}iv:2303.09521}.

\bibitem[Chu97]{Chung}
F.~R.~K. Chung.
\newblock Open problems of {P}aul {E}rd{\H{o}}s in graph theory.
\newblock {\em Journal of Graph Theory}, 25(1):3--36, 1997.

\bibitem[DJ17]{DJ}
R.~Day and R.~Johnson.
\newblock Multicolour {R}amsey numbers of odd cycles.
\newblock {\em Journal of Combinatorial Theory, Series B}, 124:56--63, 2017.

\bibitem[EG75]{EG}
P.~Erd{\H{o}}s and L.~R. Graham.
\newblock On partition theorems for finite graphs.
\newblock In {\em Infinite and finite sets (Colloq., Keszthely, 1973; dedicated
  to P. Erd{\H{o}}s on his 60th birthday) Vol I.}, volume~10. Colloq. Math.
  Soc. J\'anos Bolyai, 1975.

\bibitem[JS21]{JS}
M.~Jenssen and J.~Skokan.
\newblock Exact {R}amsey numbers of odd cycles via nonlinear optimisation.
\newblock {\em Advances in Mathematics}, 376(107444), 2021.

\bibitem[LC19]{lin-chen}
Q.~Lin and W.~Chen.
\newblock New upper bound for multicolor {R}amsey number of odd cycles.
\newblock {\em Discrete Mathematics}, 342(1):217--220, 2019.

\bibitem[Wig24]{yuvalnotes}
Y.~Wigderson.
\newblock Ramsey theory--lecture notes.
\newblock 2024.
\newblock
  \url{https://n.ethz.ch/~ywigderson/math/static/RamseyTheory2024LectureNotes.pdf}.

\end{thebibliography}

\end{document}